\documentclass[12pt]{amsart}
\usepackage{amsmath}
\usepackage{amsthm}
\usepackage{tikz}
\usepackage{tikz-cd}
\usepackage{hyperref}
\usepackage[left=2.5cm,right=2.5cm,top=3cm,bottom=3cm]{geometry}
\usepackage{eucal}
\usepackage{palatino}
\usepackage{euler}
\usepackage{mathrsfs}
\usepackage{amssymb}
\usepackage{amscd}
\usepackage{latexsym}
\usepackage{epsfig}
\usepackage{graphicx}
\usepackage{psfrag}
\usepackage{caption}
\usepackage{rotating}
\usepackage{mathtools}

\usepackage{epsfig}

\usepackage{url}
\usepackage{cite}
\usepackage{array}

\newtheorem{theorem}{Theorem}

\newtheorem{lemma}[theorem]{Lemma}

\theoremstyle{definition}
\newtheorem{definition}[theorem]{Definition}

\theoremstyle{remark}
\newtheorem{remark}{Remark}

\DeclareMathOperator{\GL}{GL}

\DeclareMathOperator{\Adj}{Ad}

\DeclareMathOperator{\adj}{ad}

\title[An equivariant description of certain holomorphic symplectic varieties]{An equivariant description of certain \\ holomorphic symplectic varieties}
\author{Peter Crooks}
\address{Institute of Differential Geometry, Gottfried Wilhelm Leibniz Universit\"{a}t Hannover, Welfengarten 1, 30167 Hannover, Germany}
\email{~~~peter.crooks@math.uni-hannover.de}

\begin{document}
	
	\begin{abstract}
This short note considers varieties of the form $G\times S_{\text{reg}}$, where $G$ is a complex semisimple group and $S_{\text{reg}}$ is a regular Slodowy slice in the Lie algebra of $G$. Such varieties arise naturally in hyperk\"ahler geometry, theoretical physics, and in the theory of abstract integrable systems developed by Fernandes, Laurent-Gengoux, and Vanhaecke. In particular, previous work of the author and Rayan uses a Hamiltonian $G$-action to endow $G\times S_{\text{reg}}$ with a canonical abstract integrable system. One might therefore wish to understand, in some sense, all examples of abstract integrable systems arising from Hamiltonian $G$-actions. Accordingly, we consider a holomorphic symplectic variety $X$ carrying an abstract integrable system induced by a Hamiltonian $G$-action. Under certain hypotheses, we show that there must exist a $G$-equivariant variety isomorphism $X\cong G\times S_{\text{reg}}$.       
	\end{abstract}
	
	\maketitle
	\section{Introduction}	
\subsection{Some preliminaries}\label{Subsection: Lie-theoretic preliminaries}
We will work exclusively over $\mathbb{C}$, understanding it as implicitly present whenever a base field is needed. Now let $G$ be a connected, simply-connected semisimple linear algebraic group having rank equal to $\mathrm{rk}(G)$, Lie algebra denoted $\mathfrak{g}$, and adjoint representation denoted $\Adj:G\rightarrow\GL(\mathfrak{g})$. Note that $\Adj$ induces the adjoint action of $G$ on $\mathfrak{g}$, whose orbits are called the \textit{adjoint orbits} of $G$. We shall let $\mathcal{O}(x)\subseteq\mathfrak{g}$ denote the adjoint orbit containing $x\in\mathfrak{g}$, i.e.
$$\mathcal{O}(x):=\{\Adj_g(x):g\in G\}.$$

The Killing form is $\mathrm{Ad}$-invariant and nondegenerate, and therefore induces an isomorphism $\mathfrak{g}\cong\mathfrak{g}^*$ between the adjoint and coadjoint representations of $G$. We will often deal with moment maps for Hamiltonian $G$-actions, which by virtue of our isomorphism $\mathfrak{g}\cong\mathfrak{g}^*$ shall always be regarded as taking values in $\mathfrak{g}$. 

Let $\adj:\mathfrak{g}\rightarrow\mathfrak{gl}(\mathfrak{g})$ be the adjoint representation of $\mathfrak{g}$. An element $x\in\mathfrak{g}$ is called \textit{regular} when the dimension of $\mathrm{ker}(\adj_x)$ coincides with $\mathrm{rk}(G)$, and we shall let $\mathfrak{g}_{\mathrm{reg}}\subseteq\mathfrak{g}$ denote the open dense subvariety of all regular elements. This subvariety is invariant under the adjoint action, and as such is a union of certain adjoint orbits --- called the \textit{regular adjoint orbits}. Equivalently, an adjoint orbit is regular if and only if its dimension is $\dim(G)-\mathrm{rk}(G)$.

Recall that $(\xi,h,\eta)\in\mathfrak{g}^{\oplus 3}$ is called an $\mathfrak{sl}_2(\mathbb{C})$-\textit{triple} if the relations
$$[\xi,\eta]=h,\quad [h,\xi]=2\xi,\quad [h,\eta]=-2\eta$$ hold in $\mathfrak{g}$, and is called a \textit{regular} $\mathfrak{sl}_2(\mathbb{C})$-\textit{triple} when we also have $\xi,\eta\in\mathfrak{g}_{\text{reg}}$. Take a regular $\mathfrak{sl}_2(\mathbb{C})$-triple $(\xi,h,\eta)$, fixed for the duration of this paper, and consider its associated \textit{Slodowy slice}
$$S_{\text{reg}}:=\xi+\mathrm{ker}(\adj_\eta):=\{\xi+x:x\in\mathrm{ker}(\adj_\eta)\}\subseteq\mathfrak{g}.$$ This slice is a $\mathrm{rk}(G)$-dimensional affine-linear subspace of $\mathfrak{g}$ enjoying the following properties: $S_{\text{reg}}\subseteq\mathfrak{g}_{\text{reg}}$ and
each regular adjoint orbit meets $S_{\text{reg}}$ in a unique point (see \cite[Thm. 8]{KostantLie}). Taken together, these two properties imply that
\begin{equation}\label{Equation: Isomorphism Slodowy slice quotient} \varphi:S_{\text{reg}}\rightarrow\mathfrak{g}_{\text{reg}}/G,\quad x\mapsto\mathcal{O}(x)\end{equation}     	  
defines an isomorphism of algebraic varieties.

\subsection{The main motivating example}
The affine variety $G\times S_{\text{reg}}$ has received some attention in the research literature. Among other things, it is known to carry a distinguished hyperk\"ahler manifold structure (see \cite{Bielawski}), and it arises as an important object in Moore and Tachikawa's discussion of certain two-dimensional topological quantum field theories (see \cite{Moore}). At the same time, this variety and its properties will feature prominently in our paper. To elaborate on this, let us use the term \textit{holomorphic symplectic variety} for a smooth algebraic variety $X$ endowed with a holomorphic symplectic form $\omega$. A left action of $G$ on $X$ shall then be called \textit{Hamiltonian} if the action is algebraic, $\omega$ is $G$-invariant, and there exists a moment map, i.e. a $G$-equivariant smooth algebraic variety morphism $\mu:X\rightarrow\mathfrak{g}$ satisfying
$$d\big(\langle\mu,\theta\rangle\big)=\iota_{\theta_X}\omega$$
for all $\theta\in\mathfrak{g}$. Here, $G$-equivariance is with respect to the adjoint action of $G$ on $\mathfrak{g}$, $\langle\cdot,\cdot\rangle$ is the Killing form on $\mathfrak{g}$, and $\theta_X$ denotes the fundamental vector field on $X$ associated to $\theta\in\mathfrak{g}$. 

It turns out $G\times S_{\text{reg}}$ is canonically a holomorphic symplectic variety, a consequence of its hyperk\"ahler structure. Moreover, $G$ acts freely on $G\times S_{\text{reg}}$ via
\begin{equation}\label{Equation: G-action} g\cdot (h,x):=(hg^{-1},x),\quad g\in G,\text{ }(h,x)\in G\times S_{\text{reg}}.\end{equation} This action is Hamiltonian with moment map 
\begin{equation}\label{Equation: Formula for mu_reg} \mu_{\text{reg}}:G\times S_{\text{reg}}\rightarrow\mathfrak{g},\quad (g,x)\mapsto-\Adj_{g^{-1}}(x)
\end{equation}  
(see \cite[Prop. 5]{Crooks}), known to be a submersion (see \cite[Prop. 6]{Crooks}). The connected components of $\mu_{\text{reg}}$'s fibres are therefore the leaves of a holomorphic foliation $\mathcal{F}_{\text{reg}}$ of $G\times S_{\text{reg}}$, and it is easily seen that these leaves are $\mathrm{rk}(G)$-dimensional. The pair $(G\times S_{\text{reg}},\mathcal{F}_{\text{reg}})$ is actually an example of an \textit{abstract integrable system} of rank equal to $\mathrm{rk}(G)$ (see \cite[Thm. 13]{Crooks}), for which we have the following definition.

\begin{definition}\label{Definition: Abstract integrable system}
Let $X$ be a holomorphic symplectic variety and $\mathcal{F}$ a holomorphic foliation of $X$ with $r$-dimensional leaves. One calls $(X,\mathcal{F})$ an \textit{abstract integrable system of rank} $r$ if each $x\in X$ has an open neighbourhood $U$, together with leaf-wise constant holomorphic functions on $U$ whose Hamiltonian vector fields span $T\mathcal{F}\subseteq TX$ on $U$.
\end{definition}

A few brief comments are in order. Firstly, Definition \ref{Definition: Abstract integrable system} is just a holomorphic counterpart of \cite[Def. 2.6]{Fernandes}, in which Fernandes, Laurent-Gengoux, and Vanhaecke introduce the notion of an \textit{abstract noncommutative\footnote{Note that Definition \ref{Definition: Abstract integrable system} suppresses the term ``noncommutative'', which is in keeping with \cite[Def. 2]{Crooks}.} integrable system} in the smooth category. Very roughly speaking, this notion aims to describe certain integrable systems in purely foliation-theoretic terms. We refer the reader to \cite{Fernandes} for further details.     

Let us return to the main discussion. In particular, note that $\mathcal{F}_{\text{reg}}$ is a foliation whose leaves are the connected components of a moment map's fibres. It is therefore natural to seek conditions under which a moment map will, analogously to $\mu_{\text{reg}}$ in the case of $(G\times S_{\text{reg}},\mathcal{F}_{\text{reg}})$, induce an abstract integrable system. To this end, we have \cite[Thm. 14]{Crooks}:         

\begin{theorem}\label{Theorem: Examples theorem}
Let $X$ be a holomorphic symplectic variety equipped with a locally free Hamiltonian $G$-action admitting $\mu:X\rightarrow\mathfrak{g}$ as a moment map. Let $\mathcal{F}_{\mu}$ denote the holomorphic foliation of $X$ whose leaves are the connected components of $\mu$'s fibres.\footnote{Since the $G$-action is locally free, $\mu$ is a submersion (see \cite[Prop. III.2.3]{Audin}) and all fibres have the same dimension.} Then the pair $(X,\mathcal{F}_{\mu})$ is an abstract integrable system if and only if $\dim(X)=\dim(G)+\mathrm{rk}(G)$ and $\mu(X)\subseteq\mathfrak{g}_{\emph{reg}}$, in which case $\mathrm{rk}(G)$ is the rank of the system. 	
\end{theorem}	

\subsection{Description of the main result}
This paper is an attempt to (at least partially) understand the class of abstract integrable systems $(X,\mathcal{F}_{\mu})$ that arise by satisfying the hypotheses of Theorem \ref{Theorem: Examples theorem}. More precisely, let $X$ be a holomorphic symplectic variety endowed with a Hamiltonian action of $G$ and moment map $\mu:X\rightarrow\mathfrak{g}$. We would like to better understand those cases in which all of the following conditions are satisfied:
\begin{itemize}
	\item the $G$-action is locally free,
	\item $\dim(X)=\dim(G)+\mathrm{rk}(G)$, and
	\item $\mu(X)\subseteq\mathfrak{g}_{\text{reg}}$.
\end{itemize}
Our main result imposes some slightly more restrictive conditions, and then completely classifies $X$ up to a $G$-equivariant variety isomorphism. In more detail, our main result is as follows.

\begin{theorem}\label{Theorem: Main theorem}
Let $X$ be a holomorphic symplectic variety endowed with a Hamiltonian $G$-action and admitting $\mu:X\rightarrow\mathfrak{g}$ as a moment map. If
\begin{itemize}
\item[(i)] $X$ is affine,	
\item[(ii)] the $G$-action is free,
\item[(iii)] $\dim(X)=\dim(G)+\mathrm{rk}(G)$,
\item [(iv)] $\mu(X)=\mathfrak{g}_{\emph{reg}}$, and
\item[(v)] $\mu^{-1}(\mathcal{O})$ is an irreducible subvariety of $X$ for all adjoint orbits $\mathcal{O}\subseteq\mathfrak{g}$, 
\end{itemize}
then there exists a $G$-equivariant variety isomorphism $X\cong G\times S_{\emph{reg}}$.
\end{theorem}

We shall devote Section \ref{Section: Proof of the main result} to the proof of this theorem. In the interim, let us make a few remarks about the hypotheses appearing in Theorem \ref{Theorem: Main theorem}. 

\begin{remark}\label{Remark: New remark}
As one might expect, $X=G\times S_{\text{reg}}$, the Hamiltonian action \eqref{Equation: G-action}, and the moment map $\mu=\mu_{\text{reg}}$ satisfy Conditions (i)--(v). The first three of these conditions are immediately seen to hold, while the fourth is satisfied by virtue of \cite[Prop. 6]{Crooks}. To verify Condition (v), let $\mathcal{O}\subseteq\mathfrak{g}$ be an adjoint orbit. Since $\mu(X)=\mathfrak{g}_{\text{reg}}$, we must have $\mu^{-1}(\mathcal{O})=\emptyset$ whenever $\mathcal{O}$ is not regular. If $\mathcal{O}$ is regular, then the isomorphism \eqref{Equation: Isomorphism Slodowy slice quotient} implies that $-\mathcal{O}:=\{-x:x\in\mathcal{O}\}$ meets $S_{\text{reg}}$ is a unique point $y$, and one can use \eqref{Equation: Formula for mu_reg} to check that $$\mu^{-1}(\mathcal{O})=G\times\{y\}\subseteq G\times S_{\text{reg}}.$$ We thus see that $\mu^{-1}(\mathcal{O})$ is irreducible for all adjoint orbits $\mathcal{O}\subseteq\mathfrak{g}$.   
\end{remark}

\begin{remark}\label{Remark: First remark}
We require $X$ to be affine in order to use \cite[Section III, Cor. 1]{Luna} in the proof of Theorem \ref{Theorem: Main theorem}. In more detail, this referenced result considers a reductive linear algebraic group $H$ acting algebraically on an affine variety $Y$ and is formulated as follows: the canonical morphism $Y\rightarrow Y/H$ is a principal $H$-bundle\footnote{We shall always use the algebro-geometric notion of a principal bundle (see \cite[Section I]{Luna}), which is defined to be \'{e}tale-locally trivial.} if and only if $H$ acts freely on $Y$. For a more situation-specific version of this result, assume that $Y$ is also irreducible. Additionally, let $Z$ be a normal variety and $f:Y\rightarrow Z$ a surjective morphism with the property that each fibre is a single $H$-orbit. One can then show that there exists a variety isomorphism $Y/H\xrightarrow{\cong} Z$ making the diagram
$$\begin{tikzcd}[column sep=small]
	& Y \arrow[dl] \arrow[dr, "f"] & \\
	Y/H \arrow{rr}{\cong} & & Z
\end{tikzcd}$$
commute (see \cite[Cor. 25.3.4 and Prop. 25.3.5]{Tauvel}), so that \cite[Section III, Cor. 1]{Luna} takes the following form: $f$ is a principal $H$-bundle if and only if $H$ acts freely on $Y$. We will later apply this rephrased version of \cite[Section III, Cor. 1]{Luna} to argue that a particular map $X\rightarrow S_{\text{reg}}$ is a principal $G$-bundle, for which we must assume that $X$ (like $Y$ above) is affine.     
\end{remark}

\begin{remark}\label{Remark: Third remark}
Theorem \ref{Theorem: Main theorem} does not hold if one relaxes Condition (ii) to require only that the $G$-action be locally free. To see this, let $Z(G)$ denote the centre of $G$. The action \eqref{Equation: G-action} of $G$ on $G\times S_{\text{reg}}$ restricts to a $Z(G)$-action, which in turn commutes with the original $G$-action. In other words, $G\times S_{\text{reg}}$ carries a Hamiltonian action of $G\times Z(G)$. Now note that $Z(G)$ is a finite group, a consequence of having taken $G$ to be semisimple. It follows that $(G\times S_{\text{reg}})/Z(G)$ is the holomorphic symplectic quotient of $G\times S_{\text{reg}}$ by $Z(G)$ (see \cite[Section 7.5]{Kobayashi} for details on holomorphic symplectic quotients). This quotient carries a residual Hamiltonian $G$-action whose moment map is obtained by letting $\mu_{\text{reg}}$ descend to the quotient $(G\times S_{\text{reg}})/Z(G)$. An examination of \eqref{Equation: G-action} reveals that this quotient is $G$-equivariantly isomorphic to $(G/Z(G))\times S_{\text{reg}}$, with $G$-acting on the first factor. The moment map on $(G/Z(G))\times S_{\text{reg}}$ is given by
$$\mu:(G/Z(G))\times S_{\text{reg}}\rightarrow\mathfrak{g},\quad ([g],x)\mapsto -\Adj_{g^{-1}}(x).$$
One can now check that $X=(G/Z(G))\times S_{\text{reg}}$, its Hamiltonian $G$-action, and the moment map $\mu$ satisfy Conditions (i), (iii), (iv), and (v), with the verification of (v) being almost identical to that given in Remark \ref{Remark: New remark}. However, note that $Z(G)$ is the $G$-stabilizer of each point in $X$. It follows that the $G$-action on $X$ is locally free but need not be free. Since $G$ acts freely on $G\times S_{\text{reg}}$, this means that $X$ need not be $G$-equivariantly isomorphic to $G\times S_{\text{reg}}$.     
\end{remark}

\begin{remark}\label{Remark: Second remark}
Let $\mathcal{O}\subseteq\mathfrak{g}$ be an adjoint orbit. Condition (ii) implies that $\mu$ is a submersion (see \cite[Prop. III.2.3]{Audin}), so that $\mu^{-1}(\mathcal{O})$ is a smooth subvariety of $X$. In particular, Condition (v) holds if and only if $\mu^{-1}(\mathcal{O})$ is connected in the Zariski topology. This is in turn equivalent to the connectedness of $\mu^{-1}(\mathcal{O})$ in the complex analytic topology (see \cite[Thm. 6.1]{Osserman}), which by virtue of $\mu$ being a submersion would hold if the fibres of $\mu$ were connected (also in the complex analytic topology). Hence, in the presence of Condition (ii), Condition (v) is weaker than $\mu$ being fibre-connected. 
	
Condition (v) turns out to be strictly weaker than fibre-connectedness, even when one considers only those $X$ and $\mu$ satisfying (i)--(iv). Indeed, recall that (i)--(v) hold for the example considered in Remark \ref{Remark: New remark}. For the same example, it turns out that $\mu$ is not fibre-connected (see \cite[Section 3.2]{Crooks}).
\end{remark}

\begin{remark}\label{Remark: Fourth remark}
Theorem \ref{Theorem: Main theorem} assumes that $\mu(X)=\mathfrak{g}_{\text{reg}}$ rather than the weaker condition $\mu(X)\subseteq\mathfrak{g}_{\text{reg}}$ discussed earlier. Indeed, the theorem no longer holds when one replaces the stronger condition with the weaker one. To see this, let $U$ be any affine open subvariety of $S_{\text{reg}}$ not isomorphic to $S_{\text{reg}}$ itself and set $X:=G\times U$. Note that $X$ is an open subvariety of $G\times S_{\text{reg}}$, so that the former inherits a holomorphic symplectic variety structure from the latter. Note also that $X$ is invariant under the $G$-action \eqref{Equation: G-action}, which together with the previous sentence implies that \eqref{Equation: G-action} defines a Hamiltonian action on $X$. The moment map is $\mu_{\text{reg}}\vert_{X}$. 
	
It is not difficult to check that $X$, its Hamiltonian $G$-action, and the moment map $\mu=\mu_{\text{reg}}\vert_{X}$ satisfy Conditions (i)--(iii) in Theorem \ref{Theorem: Main theorem}, and one can adapt the relevant part of Remark \ref{Remark: New remark} to show that Condition (v) is also satisfied. Condition (iv) does not hold, however, as one can use \eqref{Equation: Isomorphism Slodowy slice quotient}, \eqref{Equation: Formula for mu_reg}, and the fact that $U$ is a proper subvariety of $S_{\text{reg}}$ to show that $\mu(X)$ is a proper subset of $\mathfrak{g}_{\text{reg}}$. The varieties $X$ and $G\times S_{\text{reg}}$ are also not $G$-equivariantly isomorphic, since $U$ being non-isomorphic to $S_{\text{reg}}$ precludes the quotients $X/G$ ($\cong U$) and $(G\times S_{\text{reg}})/G$ ($\cong S_{\text{reg}}$) from being isomorphic.               
\end{remark}

\section{Proof of the main result}\label{Section: Proof of the main result}
Assume that the hypotheses of Theorem \ref{Theorem: Main theorem} are satisfied and define $\overline{\mu}:X\rightarrow S_{\text{reg}}$ to be the following composite map:
\begin{equation}\label{Equation: Quotient moment map}\overline{\mu}:=\left(X\xrightarrow{\mu}\mathfrak{g}_{\text{reg}}\xrightarrow{\pi}\mathfrak{g}_{\text{reg}}/G\xrightarrow{\varphi^{-1}}S_{\text{reg}}\right),\end{equation}
where $\pi$ is the quotient map and $\varphi$ is the isomorphism defined in \eqref{Equation: Isomorphism Slodowy slice quotient}. More concretely,  $\overline{\mu}$ assigns to each $x\in X$ the unique point at which $S_{\text{reg}}$ intersects $\mathcal{O}(\mu(x))$. It follows that 
\begin{equation}\label{Equation: Fibre formula}\overline{\mu}^{-1}(y)=\mu^{-1}(\mathcal{O}(y))\end{equation} for all $y\in S_{\text{reg}}$.

Now fix a point $y\in S_{\text{reg}}$. The fibre $\overline{\mu}^{-1}(y)$ is then nonempty, as Condition (iv) implies that $\overline{\mu}$ is surjective. Accordingly, we may choose a point $x\in\overline{\mu}^{-1}(y)$. At the same time, we can use \eqref{Equation: Fibre formula} and $\mu$'s $G$-equivariance property to conclude that $\overline{\mu}^{-1}(y)$ is a $G$-invariant subvariety of $X$. It follows that the $G$-orbit in $X$ through $x$, denoted $G\cdot x$, belongs to $\overline{\mu}^{-1}(y)$. 

We will establish that $G\cdot x=\overline{\mu}^{-1}(y)$. To this end, note that \eqref{Equation: Fibre formula} and Condition (v) show $\overline{\mu}^{-1}(y)$ to be irreducible. Proving $G\cdot x=\overline{\mu}^{-1}(y)$ therefore reduces to showing that $G\cdot x$ is closed and has dimension equal to that of $\overline{\mu}^{-1}(y)$. Accordingly, note that the closure of $G\cdot x$ is a union of the orbit itself and a (possibly empty) collection of strictly lower-dimensional $G$-orbits (see \cite[Prop. 21.4.5]{Tauvel}), while Condition (ii) implies that all $G$-orbits are $\dim(G)$-dimensional. These observations imply that $G\cdot x$ is closed and $\dim(G)$-dimensional. At the same time, Condition (ii) allows us to conclude that $\mu$ is a submersion (see \cite[Prop. III.2.3]{Audin}), giving rise to the following calculation:
\begin{align*}\dim(\overline{\mu}^{-1}(y))=\dim(\mu^{-1}(\mathcal{O}(y))) & = \dim(X)-\dim(g_{\text{reg}})+\dim(\mathcal{O}(y)) \\ & =(\dim(G)+\mathrm{rk}(G))-\dim(G)+(\dim(G)-\mathrm{rk}(G))\\
& = \dim(G),
\end{align*}
where we have used the fact that $\dim(\mathcal{O}(y))=\dim(G)-\mathrm{rk}(G)$, a consequence of $\mathcal{O}(y)$ being regular. 
Hence $G\cdot x=\overline{\mu}^{-1}(y)$, as desired. 

We have shown that each fibre of $\overline{\mu}$ is a single $G$-orbit, one of the hypotheses required in order to apply the version of \cite[Section III, Cor. 1]{Luna} discussed at the end of Remark \ref{Remark: First remark}. As for the other hypotheses, we know $G$ to be reductive, $X$ to be affine, $S_{\text{reg}}$ to be normal, and $G$ to act freely on $X$. Only one hypothesis remains to be checked, namely that $X$ is irreducible. To this end, the following lemma will be useful. 

\begin{lemma}\label{Lemma: Helpful}
	The map $\overline{\mu}$ is a submersion.
\end{lemma}

\begin{proof}
	Recall the definition of $\overline{\mu}$ given in \eqref{Equation: Quotient moment map}. Having noted that $\mu$ is a submersion, it will suffice to prove that $\varphi^{-1}\circ\pi$ is a submersion. To this end, suppose that $x\in\mathfrak{g}_{\text{reg}}$. By virtue of the isomorphism \eqref{Equation: Isomorphism Slodowy slice quotient}, there exist elements $g\in G$ and $y\in S_{\text{reg}}$ for which $x=\Adj_g(y)$. Now observe that
	$$\psi:S_{\text{reg}}\rightarrow\mathfrak{g}_{\text{reg}},\quad z\mapsto \Adj_g(z)$$
	is a section of $\varphi^{-1}\circ\pi$ satisfying $\psi(y)=x$. It follows that $d_x(\varphi^{-1}\circ\pi)\circ d_y\psi$ must be the identity on $T_yS_{\text{reg}}$, where $d_x(\varphi^{-1}\circ\pi)$ and $d_y\psi$ are the differential of $\varphi^{-1}\circ\pi$ at $x$ and the differential of $\psi$ at $y$, respectively. This shows $d_x(\varphi^{-1}\circ\pi)$ to be surjective, and we conclude that $\varphi^{-1}\circ\pi$ is indeed a submersion.  
\end{proof}

\begin{remark}
An alternative and perhaps more conceptual proof can be roughly sketched as follows. There exist $\mathrm{rk}(G)$ algebraically independent homogeneous generators of $\mathbb{C}[\mathfrak{g}]^G$, the algebra of $\Adj$-invariant polynomials on $\mathfrak{g}$. One can assemble these polynomials into the components of a map $\mathfrak{g}\rightarrow\mathbb{C}^{\mathrm{rk}(G)}$, called the \textit{adjoint quotient}, which is known to be a submersion when restricted to $\mathfrak{g}_{\text{reg}}$ (cf. \cite[Thm. 9]{KostantLie}). This restricted adjoint quotient and $\varphi^{-1}\circ\pi$ are related by composition with an isomorphism $\mathbb{C}^{\text{rk}(G)}\cong S_{\text{reg}}$, owing to the fact that $S_{\text{reg}}$ is a section of the adjoint quotient. It follows that $\varphi^{-1}\circ\pi$ is also a submersion, which, as noted in the proof above, is sufficient to conclude that $\overline{\mu}$ is a submersion.  
\end{remark}

Let us return to the proof of Theorem \ref{Theorem: Main theorem}. We note that the fibres of $\overline{\mu}$ are connected in the complex analytic topology, as each fibre is a $G$-orbit. Together with Lemma \ref{Lemma: Helpful}, this implies that $X$ is itself connected in the complex analytic topology. In particular, $X$ is Zariski-connected. Since $X$ is smooth, this amounts to $X$ being irreducible. 

By the discussion from the paragraph preceding Lemma \ref{Lemma: Helpful}, we may apply \cite[Section III, Cor. 1]{Luna} and conclude that $\overline{\mu}:X\rightarrow S_{\text{reg}}$ is a principal $G$-bundle. This bundle is trivial since the base $S_{\text{reg}}$ is affine space (see \cite[Thm. C]{Raghunathan} or \cite[Prop. 3.9]{Wendt}). In particular, there exists a $G$-equivariant variety isomorphism $X\cong G\times S_{\text{reg}}$.       

\subsection*{Acknowledgements} The author is grateful to Roger Bielawski for asking the questions that ultimately prompted this work, and to Steven Rayan for several fruitful discussions. The Institute of Differential Geometry at Leibniz Universit\"{a}t Hannover supported this work through the author's postdoctoral fellowship.

	\bibliographystyle{acm} 
	\bibliography{BriefArxiv}

\begin{thebibliography}{10}

\bibitem{Audin}
{\sc Audin, M.}
\newblock {\em Torus actions on symplectic manifolds}, revised~ed., vol.~93 of
  {\em Progress in Mathematics}.
\newblock Birkh\"auser Verlag, Basel, 2004.

\bibitem{Bielawski}
{\sc Bielawski, R.}
\newblock Hyperk\"ahler structures and group actions.
\newblock {\em J. London Math. Soc. (2) 55}, 2 (1997), 400--414.

\bibitem{Crooks}
{\sc Crooks, P., and Rayan, S.}
\newblock Abstract integrable systems on hyperk\"ahler manifolds arising from
  {S}lodowy slices. arxiv:1706.05819 (2017), 17 pages.

\bibitem{Fernandes}
{\sc Fernandes, R.~L., Laurent-Gengoux, C., and Vanhaecke, P.}
\newblock Global action-angle variables for non-commutative integrable systems.
  arxiv:1503.00084 (2015), 44 pages. {T}o appear in {J}. {S}ymplectic {G}eom.

\bibitem{Kobayashi}
{\sc Kobayashi, S.}
\newblock {\em Differential geometry of complex vector bundles}, vol.~15 of
  {\em Publications of the Mathematical Society of Japan}.
\newblock Princeton University Press, Princeton, NJ; Princeton University
  Press, Princeton, NJ, 1987.
\newblock Kan\^o Memorial Lectures, 5.

\bibitem{KostantLie}
{\sc Kostant, B.}
\newblock Lie group representations on polynomial rings.
\newblock {\em Amer. J. Math. 85\/} (1963), 327--404.

\bibitem{Luna}
{\sc Luna, D.}
\newblock Slices \'etales.
\newblock 81--105. Bull. Soc. Math. France, Paris, M\'emoire 33.

\bibitem{Moore}
{\sc Moore, G.~W., and Tachikawa, Y.}
\newblock On 2d {TQFT}s whose values are holomorphic symplectic varieties.
\newblock In {\em String-{M}ath 2011}, vol.~85 of {\em Proc. Sympos. Pure
  Math.} Amer. Math. Soc., Providence, RI, 2012, pp.~191--207.

\bibitem{Osserman}
{\sc Osserman, B.}
\newblock Complex varieties and the analytic topology. \\ \url{
  https://www.math.ucdavis.edu/~osserman/classes/248B-W12/notes/analytic.pdf}.

\bibitem{Raghunathan}
{\sc Raghunathan, M.~S.}
\newblock Principal bundles on affine space.
\newblock In {\em C. {P}. {R}amanujam---a tribute}, vol.~8 of {\em Tata Inst.
  Fund. Res. Studies in Math.} Springer, Berlin-New York, 1978, pp.~187--206.

\bibitem{Tauvel}
{\sc Tauvel, P., and Yu, R. W.~T.}
\newblock {\em Lie algebras and algebraic groups}.
\newblock Springer Monographs in Mathematics. Springer-Verlag, Berlin, 2005.

\bibitem{Wendt}
{\sc Wendt, M.}
\newblock Rationally trivial torsors in {$\Bbb A^1$}-homotopy theory.
\newblock {\em J. K-Theory 7}, 3 (2011), 541--572.

\end{thebibliography}
\end{document}